\documentclass[12pt,a4paper]{amsart}

\usepackage{amsthm, amssymb}
\usepackage{mathrsfs, multirow}
\usepackage{fancyhdr}
\usepackage{indentfirst}
\usepackage{cmap}

\usepackage{tikz, tikz-3dplot}
\usetikzlibrary{calc, 3d, patterns}

\usepackage{color}\definecolor{darkblue}{rgb}{0,0.1,.5}
\usepackage[colorlinks=true,linkcolor=darkblue, urlcolor=darkblue, citecolor=darkblue]{hyperref}
\usepackage[left=2cm, right=2cm, top=3cm, bottom=3cm]{geometry}
\def\R{\mathbb{R}}
\def\C{\mathbb{C}}
\def\Z{\mathbb{Z}}
\def\sk{\mathcal K}
\def\sj{\mathcal J}
\def\zk{\mathcal {Z_K}}	
\def\ve{\varepsilon}

\newtheorem{theorem}{Theorem}[section]
\newtheorem{lemma}[theorem]{Lemma}
\newtheorem{prop}[theorem]{Proposition}

\renewenvironment{proof}[1][{\itshape Proof}]{{\itshape #1. }}{\qed}

\theoremstyle{definition}
\newtheorem{defn}[theorem]{Definition}
\newtheorem{construction}{Construction}
\newtheorem{exm}[theorem]{Example}

\author{Semyon Abramyan}
\title{Iterated Higher Whitehead products in topology of moment-angle complexes}
\address{Faculty of Mathematics, National Research University HSE, 6 Usacheva Street, Moscow, RUSSIA}
\email{semyon.abramyan@gmail.com}

\begin{document}

\maketitle

\tableofcontents

\section{Introduction}
The moment-angle complex $\zk$ is a cell complex built from products of polydisks and tori parametrized by simplices in finite simplicial complex $\sk$. There is torus action on $\zk$, which playes an important role in toric topology (see~\cite{TT}). In the case when $\sk$ is a triangulation of a sphere, $\zk$ is a topological manifold with rich geometric structure. Moment-angle complexes are particular examples of the homotopy-theoretical construction of \textit{polyhedral products}, which provides a wonderful "testing ground" for application of unstable homotopy theory techniques.

In this paper we study the topological structure of moment-angle complexes $\zk$. Interest to higher Whitehead products in homotopy groups of moment-angle complexes and polyhedral products goes back to the work of Panov and Ray \cite{pr} where first structure results were obtained. Important results about structure of higher Whitehead products were obtained in the works of Grbi\'c, Theriault
\cite{gt} and Iriye, Kishimoto \cite{ki}.

We consider two classes of simplicial complexes. The first class $B_{\Delta}$ consists of simplicial complexes $\sk$ for which $\zk$ is homotopy equivalent to a wedge of spheres. The second class $W_{\Delta}$ consists of $\sk\in B_{\Delta}$ such that all spheres in the wedge are realized by iterated higher Whitehead products (see~\hyperref[defenitions]{\S\ref{defenitions}}). Buchstaber and Panov asked in \cite[Problem~8.4.5]{TT} if it is true that $B_{\Delta} = W_{\Delta}$. In this paper we show that this is not the case (see~\hyperref[non]{\S\ref{non}}). Namely, we give an example of a simplicial complex whose corresponding moment-angle complex is homotopy equivalent to a wedge of spheres, but there is a sphere which cannot be realized by any linear combination of iterated higher Whitehead products.

On the other hand we show that class $W_{\Delta}$ is large enough. Namely, we show that the class $W_{\Delta}$ is closed with respect to two explicitly defined operations on simplicial complexes (see~\hyperref[glue]{Proposition~\ref{glue}} and \hyperref[soper]{Th.~\ref{soper}}). Then using these operations we prove that there exists a simplicial complex that realizes any given iterated higher Whitehead product (see~\hyperref[exist]{Theorem~\ref{exist}}). We also describe the smallest simplicial complex that realizes an iterated product with only two pairs of nested brackets (see~\hyperref[least]{Theorem~\ref{least}}).

I wish to express gratitude to my supervisor Taras E.~Panov for stating the problem, valuable advice and stimulating discussions. I thank Anton Ayzenberg for valuable comments. I also grateful to express my thanks to Maria Zarubina for careful reading of the text and correcting typos.

\section{Preliminaries}
\label{defenitions}

A \textit{simplicial complex} $\sk$ on the set $[m] \overset{\rm def}{=} \{1, 2, \ldots, m\}$ is a collection of subsets $I \subset [m]$ closed under taking any subsets. We refer to $I\in\sk$ as \textit{simplices} or \textit{faces} of $\sk$, and always assume that $\varnothing\in\sk$. Denote by $\Delta^{m-1}$ or $\Delta(1, \ldots, m)$ the full simplex on the set $[m]$.

Assume we are given a set of $m$ topological pairs
\begin{align*}
	(\underline X, \underline A) = \{(X_1, A_1),\dots, (X_m, A_m)\}
\end{align*}
where $A_i\subset X_i$. For each simplex $I\in\sk$ we set
\begin{align*}
	(\underline X, \underline A)^I = \{(x_1, \ldots, x_m)\in X_1\times\cdots\times X_m\;|\; x_j \in A_j \text{ for $j\notin I$}\}.
\end{align*}
The \textit{polyhedral product} of $(\underline X, \underline A)$ corresponding to $\sk$ is the following subset of $X_1\times\dots\times X_m$:
	\begin{align*}
	(\underline X, \underline A)^{\sk} = \bigcup\limits_{I\in \sk} (\underline X, \underline A)^I\qquad (\subset X_1\times\cdots\times X_n).
	\end{align*}
	
In case when all the pairs $(X_i, A_i)$ are $(D^2, S^1)$ we use the notation $\zk$ for $(\underline X, \underline A)^{\sk}$, and refer to $\zk = (D^2, S^1)^{\sk}$ as the \textit{moment-angle complex}. Also, if all the pairs $(X_i, A_i)$ are $(X, pt)$ we use the abbreviated notation $X^{\sk}$ for $(\underline X, \underline A)^{\sk}$.

\begin{theorem}[{\cite[Ch.4]{TT}}]\label{cofib}
	The moment-angle complex $\zk$ is the homotopy fibre of the canonical inclusion $(\C P^{\infty})^{\sk}\hookrightarrow (\C P^{\infty})^m$.
\end{theorem}

We shall also need the following more explicit description of the map $\mathcal Z_{\sk} \to (\C P^{\infty})^{\sk}$. Consider the map of pairs $(D^2, S^1) \to (\C P^{\infty}, pt)$ sending the interior of the disc homeomorphically onto the complement of the basepoint in $\C P^1$. By the functoriality, we have the induced map of the polyhedral products $\mathcal{Z_K} = (D^2, S^1)^{\sk} \to (\C P^{\infty})^{\sk}$.

The general definition of the higher Whitehead product can be found
in~\cite{ha}. We only describe Whitehead products in the space
$(\C P^{\infty})^{\sk}$ and their lifts to $\zk$. In this case
the indeterminacy of higher Whitehead products can be controlled
effectively because extension maps can be chosen canonically.

Let $\mu_i$ be the map
$(D^2, S^1) \to S^2 \cong \C P^1 \hookrightarrow (\C P^{\infty})^{\vee m}
{}\hookrightarrow (\C P^{\infty})^{\sk}$.
Here the second map is the canonical
inclusion of $\C P^1$ into the $i$-th summand of wedge sum. The third map is induced
by embedding of $m$ disjoint points into $\sk$.
The \emph{Whitehead product} (or \emph{Whitehead bracket}) $[\mu_i, \mu_j]$ of $\mu_i$ and $\mu_j$
is the homotopy class of the map
$$
S^3 \cong \partial D^4 \cong \partial (D^2\times D^2) \cong D^2\times S^1\cup{}
S^1\times D^2 \xrightarrow{[\mu_i,\mu_j]} (\C P^{\infty})^{\sk},
$$
where 
$$
[\mu_i, \mu_j](x, y) =
\begin{cases}
\mu_i(x) \quad \text{for $(x,y)\in D^2\times S^1$};\\
\mu_j(y) \quad \text{for $(x,y)\in S^1\times D^2$}.
\end{cases}   
$$

Every Whitehead product becomes trivial after composing with the embeding
$(\C P^{\infty})^{\sk}\hookrightarrow (\C P^{\infty})^m \simeq K(\Z^m, 2)$.
This implies that the map
$[\mu_i, \mu_j]\colon S^3 \to (\C P^{\infty})^{\sk}$ has a lift $S^3 \to \zk$;
we will use the same notation for it. Such a lift $[\mu_i, \mu_j]$
is given by the inclusion of subcomplex
$$
[\mu_i, \mu_j]\colon S^3 \cong{}
D^2\times S^1\cup S^1\times D^2 \hookrightarrow \zk.
$$
If the Whitehead product $[\mu_i, \mu_j]$
is trivial then the map $[\mu_i, \mu_j]\colon S^3 \to \zk$
can be extended canonically to a map
$D^4 \cong D^2_i\times D^2_j \hookrightarrow\zk$. 

\emph{Higher Whitehead products} are defined inductively as follows.

Let $\mu_{i_1},\dots,\mu_{i_n}$ be a collection of maps such
that the $(n-1)$-fold product
$[\mu_{i_1},\dots,\widehat{\mu_{i_k}},\dots, \mu_{i_n}]$ is trivial for any $k$.
Then for every $(n-1)$-fold product there exists a \emph{canonical}
extension $\overline{[\mu_{i_1},\dots,\widehat{\mu_{i_k}},\dots, \mu_{i_n}]}$
to a map from $D^{2(n-1)}$ which is the composition
$$
\overline{[\mu_{i_1},\ldots, \widehat{\mu_{i_k}},\ldots, \mu_{i_n}]}\colon
D^2_{i_1}\times\ldots\times D^2_{i_{k-1}}\times D^2_{i_{k+1}}\times\ldots\times
D^2_{i_n}\hookrightarrow \zk \to (\C P^{\infty})^{\sk},
$$ 
and all these extensions are compatible on the intersections.
The \emph{$n$-fold product} $[\mu_{i_1},\ldots, \mu_{i_n}]$ is defined
as the homotopy class of the map\footnote{In all set-theoretic formulae in this paper we consider the product operation to be a higher priority than the union.}
$$
S^{2n-1} \cong \partial(D^2_{i_1}\times\ldots\times D^2_{i_n})\cong
\bigcup\limits_{k=1}^n D^2_{i_1}\times\ldots\times S^1_{i_k}\times\ldots
\times D^2_{i_n}\xrightarrow{[\mu_{i_1},\ldots, \mu_{i_n}]}
(\C P^{\infty})^{\zk},
$$
which is given as follows:
\begin{align*}
[\mu_{i_1},\ldots, \mu_{i_n}](x_1,\ldots, x_n) =
\begin{cases}
\overline{[\mu_{i_1},\ldots, \mu_{i_{n-1}}]}(x_1,\ldots, x_{n-1})
&\text{for $x_n \in S^1_{i_n}$;}\\
\ldots \\
\overline{[\mu_{i_1},\ldots,\widehat\mu_{i_k},\ldots, \mu_{i_n}]}
(x_1,,\ldots,\widehat x_k,\ldots, x_n) &\text{for $x_k\in S^1_{i_k}$;}\\
\ldots \\
\overline{[\mu_{i_2},\ldots, \mu_{i_n}]}(x_2,\ldots, x_n)
&\text{for $x_1 \in S^1_{i_1}$.}
\end{cases}
\end{align*} 

Alongside with higher Whitehead products of canonical maps $\mu_i$
we will consider \emph{general iterated} higher 
Whitehead products, i.\,e. higher Whitehead products whose
arguments can be higher Whitehead products.
For example, $[\mu_1,\mu_2,[\mu_3,\mu_4,\mu_5],
[\mu_6,\mu_{13},[\mu_7,\mu_8,\mu_9],\mu_{10}],[\mu_{11},\mu_{12}]]$.
In most cases we consider only \emph{nested} iterated higher
Whitehead products, i.\,e. products of the form
$$
w = \left[\mu_{i_{01}}, \ldots, \mu_{i_{0p_0}}, \Big[\ldots \big[\mu_{i_{q1}},\ldots, \mu_{i_{qp_q}}\big]\Big]\right] \colon S^{d(w)}\to(\C P^{\infty})^{\sk}.
$$
Here $d(w)$ denotes the dimension of $w$.

As in the case of ordinary Whitehead products any iterated higher Whitehead product lifts to a map $S^{d(w)} \to\zk$ for dimensional reasons.    

\begin{defn}\label{wdelta} We say that a simplicial complex $\sk$ \textit{realizes} a Whitehead product $w$ if $\zk$ is homotopy equivalent to a wedge of spheres in which one of the wedge summands is realized by a lift $S^{d(w)}~\to~\zk$~of~$w$.

Denote by $W_{\Delta}$ the class of simplicial complexes $\sk$ such that $\zk$ is a wedge of spheres and each sphere in the wedge is a lift of a linear combination of iterated higher Whitehead products. The class $W_{\Delta}$ is not empty as it contains the boundary of simplex $\partial\Delta^n$ for each $n>0$. The moment-angle complex $\zk$ corresponding to $\partial \Delta^n$ is homotopy equivalent to $S^{2n+1}$, which can be realized by the product $[\mu_0, \ldots, \mu_n]$.
\end{defn}

We consider the following decomposition of the disc $D^2$ into 3 cells: the point $1\in D^2$ is the 0-cell; the complement to $1$ in the boundary circle is the 1-cell, which we denote by~$S$; and the interior of $D^2$ is the 2-cell, which we denote by~$D$. These cells are canonically oriented as subsets of $\R^2$. By taking products we obtain a cellular decomposition of~$(D^2)^m$ whose cells are parametrised by pairs of subsets $J,I\subset [m]$ with $J\cap I=\varnothing$: the set $J$ parametrises the $S$-cells in the product and $I$ parametrizes the $D$-cells as we describe below. We denote the cell of $(D^2)^m$
corresponding to a pair $J,I$ by $\chi(J,I)$:
\begin{align*}
	\chi(J, I) = \big\{(x_1, \ldots, x_m) \in (D^2)^m\;\big|\;\text{$x_i \in D$ for $i\in I$, $x_j \in S$ for $j \in J$ and $x_l = 1$ for $l\notin J\cup I$} \big\}.
\end{align*}
Then $\zk$ embeds as a cellular subcomplex in~$(D^2)^m$; we have $\chi(J,I)\subset\zk$ whenever $I\in\sk$.

The coproduct in the homology of a cell-complex $X$ can be defined as follows.
Consider the composite map of cellular cochain complexes
\begin{align}\label{diag_comp}
\mathcal C_*(X) \xrightarrow{\widetilde\Delta_*} \mathcal C_*(X\times X) \xrightarrow{P} \mathcal C_*(X)\otimes \mathcal C_*(X).
\end{align}
Here the map $P$ sends the basis chain corresponding to a cell $e^i\times e^j$ to $e^i\otimes e^j$. The map $\widetilde\Delta_*$ is induced by a cellular map $\widetilde\Delta$ homotopic to the diagonal $\Delta\colon X \xrightarrow{x \mapsto (x, x)} X\times X$. In homology, the map~\eqref{diag_comp} induces a coproduct $H_*(X)\to H_*(X)\otimes H_*(X)$ which does not depend on a choice of cellular approximation and is functorial. However, the map~\eqref{diag_comp} itself is not functorial because the choice of a cellular approximation is not canonical.

Nevertheless, in the case $X=\zk$ we can use the following construction.

\begin{construction}[{\cite[Construction 4.5.2]{TT}}]
\label{cell_approx}
Consider the map $\widetilde\Delta\colon D^2 \to D^2\times D^2$ given in the polar coordinates $z = \rho e^{i\varphi}\in D^2, 0\leqslant\rho\leqslant 1, 0\leqslant\varphi\leqslant  2\pi$, by the formula
\begin{align*}
	\rho e^{i\varphi}\mapsto
	\begin{cases}
	(1 - \rho + \rho e^{2i\varphi}, 1) \quad \text{for $0 \leqslant \varphi \leqslant \pi,$} \\
	(1, 1 - \rho + \rho e^{2i\varphi}) \quad \text{for $\pi  \leqslant\varphi<2\pi$}.
	\end{cases}
\end{align*}
	This is a cellular map homotopic to the diagonal $\Delta\colon D^2 \to D^2\times D^2$, and its restriction to the boundary circle $S^1$ is a diagonal approximation for $S^1$:
	\begin{center}
		\begin{tikzpicture}[every node/.style={midway}]
		\matrix[column sep={4em,between origins},row sep={4em}] at (0,0) {
			\node(A)   {$S^1$};              && \node(B) {$D^2$}; \\
			\node(C)   {$S^1\times S^1$}; && \node(D) {$D^2\times D^2.$};  \\};
		\draw[->] (A) -- (B) node[anchor=south] {};
		\draw[->] (A) -- (C) node[anchor=east]  {$\widetilde\Delta $};
		\draw[->] (B) -- (D) node[anchor=west]  {$\widetilde\Delta $};
		\draw[->](C) -- (D) node[anchor=south] {};
		\end{tikzpicture}
	\end{center}
	Taking $m$-fold product we obtain a cellular approximation $\widetilde\Delta^m\overset{\rm def}{=} \widetilde\Delta\times\dots\times \widetilde\Delta\colon (D^2)^m \to (D^2)^m \times (D^2)^m$ which restricts to a cellular approximation of the diagonal map of $\zk$, as described in the following diagram:
	\begin{center}
		\begin{tikzpicture}[every node/.style={midway}]
		\matrix[column sep={4em,between origins}, row sep={4em}] at (0, 0) {
			\node(A)   {$\mathcal{Z_K}$};         && \node(B) {$(D^2)^m$}; \\
			\node(C)   {$\mathcal{Z_K} \times \mathcal{Z_K}$}; && \node(D) {$(D^2)^m\times (D^2)^m.$};  \\};
		\draw[->] (A) -- (B) node[anchor=south] {};
		\draw[->] (A) -- (C) node[anchor=east]  {$\widetilde\Delta^m|_{\zk}$};
		\draw[->] (B) -- (D) node[anchor=west]  {$\widetilde\Delta^m$};
		\draw[->] (C) -- (D) node[anchor=south] {};
		\end{tikzpicture}
	\end{center}
The diagonal approximation $\widetilde\Delta$ is functorial with respect to maps of moment-angle complexes induced by simplicial maps.

Further on, we denote $\widetilde\Delta^m|_{\zk}$ simply by $\widetilde\Delta$. 
\end{construction}	

\section{Algebraic constructions}
\label{algebraic}
Let $\Lambda\langle u_1, \dots, u_m \rangle$ and $\Z\langle\sk\rangle$ denote respectively the exterior coalgebra and the Stanley--Reisner coalgebra of a simplicial complex $\sk$,
which is a subcoalgebra of the symmetric coalgebra $\Z\langle v_1, \dots, v_m\rangle$.
The Stanley--Reisner coalgebra $\Z\langle\sk\rangle$ is generated as a $\Z$-module by monomials whose support is a simplex of $\sk$ \cite[\S 8.4]{TT}. 
Consider the submodule $\mathcal{R_*(K)}$ of $\Lambda\langle u_1, \dots, u_m \rangle\otimes \Z\langle\sk\rangle$ additively generated by monomials not containing $u_iv_i$ and $v_i^2$.
Clearly, $\mathcal{R_*(K)}\subset\Lambda\langle u_1, \dots, u_m \rangle\otimes \Z\langle\sk\rangle$ is a subcoalgebra. We endow it with the differential $\partial = \sum\limits_{i = 1}^m u_i\frac{\partial}{\partial v_i}$ of degree $-1$.

The following statements are obtained by dualization of the corresponding statements from \cite[\S 4.5]{TT} for cellular cochains and cohomology.

\begin{lemma}\label{isochain}
	The map
	\begin{align*}
		g \colon \mathcal{R_*(K)} \xrightarrow{u_Jv_I \mapsto \chi(J, I)} \mathcal {C_*(Z_K)}
	\end{align*}
	is an isomorphism of chain complexes. Hence, there is an additive isomorphism $H\big( \mathcal{R_*(K)}, \partial\big) \cong H_*(\mathcal{Z_K})$.
\end{lemma}

\begin{lemma}
The cellular chain coalgebra $\mathcal C_*(\mathcal{Z_K})$ with the product defined via the diagonal approximation $\widetilde\Delta\colon \mathcal{Z_K} \to \mathcal{Z_K}\times\mathcal{Z_K}$ \textup{(see~\hyperref[cell_approx]{Construction~\ref{cell_approx}})} is isomorphic to the coalgebra $\mathcal{R_*(K)}$. Hence, we have an isomorphism of homology coalgebras
	\begin{align*}
		H\big(\mathcal{R_*(K)}, \partial\big) \cong H_*(\mathcal{Z_K}).
	\end{align*}
\end{lemma}

For each subset $J\subset [m]$ and for each simpicial complex on $[m]$ denote by $\sk_J$ \emph{the full subcomplex} on the vertex set $J$, i.\,e. $\sk_J = \{I\in \sk|\, I\subset J\}.$

\begin{theorem}\label{homology}
The homomorphisms
\begin{align*}
	C_{p-1}(\sk_J) \xrightarrow{L\mapsto \ve(L,J)\chi(J\backslash L, L)} \mathcal C_{p+|J|}(\mathcal{Z_K})
\end{align*}
induce injective homomorphisms
\begin{align*}
	\widetilde H_{p-1}(\sk_J) \hookrightarrow H_{p+|J|}(\mathcal{Z_K}),
\end{align*}
which are functorial with respect to simplicial inclusions. Here $L\in \sk_J$ is a simplex and $\ve(L,J)$ is the sign of the shuffle $(L,J)$.
The inclusions above induce an isomorphism ofthe  abelian groups
\begin{align*}
	h\colon \bigoplus\limits_{J\subset [m]}\widetilde H_*(\sk_J) \xrightarrow{\cong} H_*(\mathcal{Z_K}).
\end{align*}	
\end{theorem}

\section{Hurewicz homomorphism for moment-angle complexes}
\label{hurewicz}

We shall use the notation $S_i$ and $D_i$ for the $1$-cell and the $2$-cell in the $i$-th factor of $(D^2)^m$. We denote product cells in $(D^2)^m$ by the words like $D_iS_jD_k$.

In this section we consider iterated higher Whitehead products of the following form:
$$
\Big[\mu_{i_{01}}, \ldots, \mu_{i_{0p_0}}, \big[ \ldots [\mu_{i_{n 1}}, \ldots, \mu_{i_{n p_n}}] \ldots\big]\Big]: S^{2(p_0+\cdots+p_n) - (n+1)}\to \zk,
$$
where $i_{kl}\in [m]$ for all $k$ and $l$.
The existence of a simplicial complex $\sk~\in~W_{\Delta}$ realizing each product above will be proved in the next section (see~\hyperref[exist]{Theorem~\ref{exist}}).

The following lemma is a generalization of Lemma 3.1 from \cite{veryovkin16}.

\begin{lemma}
\label{hur_im_2}
	The Hurewicz image
	\begin{align*}
		h\left(\left[\mu_{i_{01}}, \ldots, \mu_{i_{0p_0}}, \left[ \ldots, \big[\mu_{i_{(n-1)1}}, \ldots, \mu_{i_{(n-1) p_{n-1}}}, [\mu_{i_{n 1}}, \ldots, \mu_{i_{n p_n}}]\big] \ldots\right]\right]\right) \in H_{2(p_0 + \cdots + p_n) - (n+1)}(\mathcal {Z_K})
	\end{align*}
	is represented by the cellular chain
	\begin{align}
	\label{f_hur_im_2}
		\prod\limits_{k = 0}^n \left(\sum\limits_{j = 1}^{p_k} D_{i_{k1}}\cdots D_{i_{k(j-1)}}S_{i_{kj}}D_{i_{k(j+1)}}\cdots D_{i_{kp_k}} \right).
 \end{align}
\end{lemma}

\begin{proof} We induct on the number $n$ of nested higher products.

	For $n = 0$ we have a single higher product
	\begin{align*}
	[\mu_1, \ldots, \mu_k]\colon S^{2k - 1} &{}\cong{} \partial (D_1^2\times\cdots\times D_k^2) \\ &{}\cong{} D_1^2\times\cdots\times D_{k-1}^2\times S_k^1 \cup \cdots \cup S_1^1\times D_2^2 \times\cdots\times D_k^2 \to (\C P^{\infty})^{\sk}
	\end{align*}
	which lifts to the inclusion of a subcomplex
	\begin{align*}
		[\mu_1, \ldots, \mu_k]\colon  D_1^2\times\cdots\times D_{k-1}^2\times S_k^1 \cup \cdots \cup S_1^1\times D_2^2 \times\cdots\times D_k^2 \to \mathcal {Z_K}.
	\end{align*}
	Therefore, the Hurewicz image is represented by the cellular chain~\eqref{f_hur_im_2}.

	Let $n = 1$, i.e. we have a product of the form $\left[\mu_{i_1}, \ldots \mu_{i_p}, [\mu_{j_1}, \ldots, \mu_{j_q}]\right], q > 1$.
	
By definition of the higher Whitehead products, the class
\[
\left[\mu_{i_1}, \ldots \mu_{i_p}, [\mu_{j_1}, \ldots, \mu_{j_q}]\right] \in \pi_{2(p+q-1)}((\C P^{\infty})^{\mathcal K})
\]
is represented by the composite map
\begin{align*}
	&S^{2(p+q-1)} \\ {}\cong{}
	&D_{i_1}^2\times\cdots\times D_{i_p}^2\times \partial D_{j_1\ldots j_q}^{2q-1}\cup\left(\left(\bigcup\limits_{k = 1}^p D_{i_1}^2\times\cdots\times D_{i_{k-1}}^2\times S_{i_k}^1\times D_{i_{k+1}}^2 \times\cdots\times D_{i_p}^2\right)\times D_{j_1\ldots j_q}^{2q-1}\right)
	\to \\ {}\to{}
	&D_{i_1}^2\times\cdots\times D_{i_p}^2\times pt \cup\left(\left(\bigcup\limits_{k = 1}^p D_{i_1}^2\times\cdots\times D_{i_{k-1}}^2\times S_{i_k}^1\times D_{i_{k+1}}^2 \times\cdots\times D_{i_p}^2\right)\times S_{j_1\ldots j_q}^{2q-1}\right)
	\to \\ {}\to{}
	&S_{i_1}^2\times\cdots\times S_{i_p}^2\times pt \cup\left(\left(\bigcup\limits_{k = 1}^p S_{i_1}^2\times\cdots\times S_{i_{k-1}}^2\times pt \times S_{i_{k+1}}^2 \times\cdots\times S_{i_p}^2\right)\times S_{j_1\ldots j_q}^{2q-1}\right)
	\to
	(\C P^{\infty})^{\mathcal K}.
\end{align*}
	
By the argument following \hyperref[cofib]{Theorem~\ref{cofib}}, the composite map above lifts to an inclusion of a cell subcomplex as described by the following diagram:
\begin{center}
	\begin{tikzpicture}[every node/.style={midway}]
	\label{d_hur_im_2}
	\matrix[column sep={4em,between origins},row sep={4em}] at (0, 0) {
		&&\node(A)   {$\mathcal {Z_K}$}; &&&& \node(B) {$(\C P^{\infty})^{\sk}$}; \\
		&&&& \node(C)   {$D_{i_1}^2\times\cdots\times D_{i_p}^2\times pt \cup\left(\left(\bigcup\limits_{k = 1}^p D_{i_1}^2\times\cdots\times S_{i_k}^1\times\cdots\times D_{i_p}^2\right)\times \left(\bigcup\limits_{k = 1}^q D_{j_1}^2\times\cdots\times S_{j_k}^1\times\cdots\times D_{j_q}^2\right)\right)$}; \\
		&&&& \node(D)   {$D_{i_1}^2\times\cdots\times D_{i_p}^2\times pt \cup\left(\left(\bigcup\limits_{k = 1}^p D_{i_1}^2\times\cdots\times S_{i_k}^1\times\cdots\times D_{i_p}^2\right)\times S_{j_1\ldots j_q}^{2q-1}\right).$};
		\\};
	\draw[->] (A) -- (B) node[anchor=south] {};
	\draw[dashed,  ->] (C) -- (A) node[anchor=west]  {};
	\draw[->] (C) -- (B) node[anchor=west]  {$\quad\left[\mu_{i_1}, \ldots \mu_{i_p}, \left[\mu_{j_1}, \ldots, \mu_{j_q}\right]\right]$};
	\draw[double distance=3pt] (C) -- (D) node[anchor=west]  {};
	\end{tikzpicture}
\end{center}
	In the union above, all cells except $D_{i_1}^2\times\cdots\times D_{i_p}^2\times pt$ have dimension $2(p+q-1)$ while $D_{i_1}^2\times\cdots\times D_{i_p}^2\times pt$ has dimension $2p < 2(p+q-1)$. Thus, the Hurewicz image of the lifted iterated higher Whitehead product in $H_{2(p+q-1)}(\mathcal{Z_K})$ is represented by the cellular chain
	\begin{align*}
		\left(\sum\limits_{k = 1}^{p} D_{i_1}\cdots D_{i_{k-1}}S_{i_k}D_{i_{k + 1}}\cdots D_{i_p} \right)\cdot\left(\sum\limits_{k = 1}^{q} D_{j_1}\cdots D_{j_{k-1}}S_{j_k}D_{j_{k + 1}}\cdots D_{j_p} \right).
	\end{align*}
	
	A similar argument applies in the general case. Consider the iterated higher Whitehead product $[\mu_{i_{01}}, \ldots, \mu_{i_{0p_0}}, w]\colon S^{2p_0 + d(w) - 1} \to (\C P^{\infty})^{\sk}$. Here
	\begin{align*}
		w = \left[\mu_{i_{11}}, \ldots, \mu_{i_{1p_1}}, \left[ \ldots, \big[\mu_{i_{(n-1)1}}, \ldots, \mu_{i_{(n-1) p_{n-1}}}, [\mu_{i_{n 1}}, \ldots, \mu_{i_{n p_n}}]\big] \ldots\right]\right].
	\end{align*}
	The sphere $S^{2p_0 + d(w) - 1}$ is decomposed into a union
	\begin{align*}
		D_{i_{01}}^2\times\cdots\times D_{i_{0p_0}}^2\times\partial D^{d(w)}\cup\left(\left(\bigcup\limits_{k = 1}^p D_{i_{01}}^2\times\cdots\times D_{i_{0(k-1)}}^2\times S_{i_{0k}}^1\times D_{i_{0(k+1)}}^2 \times\cdots\times D_{i_{0p_0}}^2\right)\times D_{w}^{d(w)}\right).
	\end{align*}
	By contracting $\partial D^{d(w)}$ to a point we obtain a cell subcomplex $X\subset \zk$. The inclusion $X\hookrightarrow\zk$ is a lift of the Whitehead product $[\mu_{i_{01}}, \ldots, \mu_{i_{0p_0}}, w]$. Arguing by induction we obtain that the Hurewicz image of the map $S^{d(w)}_w \to \mathcal{Z_K}$ is represented by the cellular chain
	\begin{align*}
	\prod\limits_{k = 1}^n \left(\sum\limits_{j = 1}^{p_k} D_{i_{k1}}\cdots D_{i_{k(j-1)}}S_{i_{kj}}D_{i_{k(j+1)}}\cdots D_{i_{kp_k}} \right).
	\end{align*}
	By dimensional reasons, $h([\mu_{i_{01}}, \ldots, \mu_{i_{0p_0}}, w]|_{D_{i_{01}}^2\times\cdots\times D_{i_{0p_0}}^2\times pt}) = 0 \in H_{2p_0 + d(w) - 1}(\mathcal{Z_K})$. Thus,
	\begin{align*}
	h([\mu_{i_{01}}, \ldots, \mu_{i_{0p_0}}, w]) = \prod\limits_{k = 0}^n \left(\sum\limits_{j = 1}^{p_k} D_{i_{k1}}\cdots D_{i_{k(j-1)}}S_{i_{kj}}D_{i_{k(j+1)}}\cdots D_{i_{kp_k}} \right).
	\end{align*}
\end{proof}

\begin{exm}
	\label{exam1}
	Consider the Whitehead product $\left[\mu_1, \mu_2, [\mu_3, \mu_4, \mu_5]\right]$. A simplicial complex $\sk \in W_{\Delta}$ realizing this product is described in \hyperref[j_0]{Example~\ref{j_0}} below. By \hyperref[hur_im_2]{Lemma~\ref{hur_im_2}}, the homology class $h\left(\left[\mu_1, \mu_2, [\mu_3, \mu_4, \mu_5]\right]\right) \in H_8(\mathcal{Z_K})$ is represented by the cellular chain
	\begin{align*}
		&D_1S_2D_3D_4S_5 + D_1S_2D_3S_4D_5 + D_1S_2S_3D_4D_5 \\{}+{}&S_1D_2D_3D_4S_5 + S_1D_2D_3S_4D_5 + S_1D_2S_3D_4D_5.
	\end{align*}
\end{exm}

\section{Operations on $W_{\Delta}$ and realization of Whitehead products}\label{sec_soper}

Let $\sk = \sk_1\cup_I\sk_2$ be a simplicial complex obtained by gluing $\sk_1$ and $\sk_2$ along a common face $I$
(we allow $I= \varnothing$, in which case $\sk_1\cup_I\sk_2 = \sk_1\sqcup\sk_2$).

\begin{prop}[compare~{\cite[Th.~8.2.1]{TT}}]\label{glue}
If $\sk_1,\sk_2\in W_{\Delta}$ then $\sk = \sk_1\cup_I\sk_2 \in W_{\Delta}$ for any common face $I$ of simplicial complexes $\sk_1$ and $\sk_2$.
\end{prop}
\begin{proof}
Consider a full subcomplex $\sk_J$ in $\sk = \sk_1\cup_I\sk_2$. Let $V(\sk_1)$, $V(\sk_2)$ be the sets of vertices of the simplicial complexes $\sk_1, \sk_2$. Denote $J_1 = V(\sk_1)\cap J, J_2 = V(\sk_2)\cap J$. Consider two cases. If $J\cap I = \varnothing$, then $\sk_J$ is the disjoint union $\sk_{J_1}\sqcup \sk_{J_2}$, and we have
\begin{align*}
   	H_p(\sk_J) \cong H_p(\sk_{J_1}) \oplus H_p(\sk_{J_2}).
\end{align*}
If $J \cap I \neq \varnothing$ then $\sk_J$ is homotopy equivalent to the wedge $\sk_{J_1}\vee \sk_{J_2}$, and
\begin{align*}
    \widetilde H_p(\sk_J) \cong \widetilde H_p(\sk_{J_1}) \oplus \widetilde H_p(\sk_{J_2}).
\end{align*}
In both cases generators of each summand map to generators of the corresponding homology groups $H_*(\mathcal Z_{\sk})$ under the homomorphism of \hyperref[homology]{Theorem~\ref{homology}}.
	
Let $\{\sigma_{\alpha}(J_1, p)\}_{\alpha \in A}, \{\sigma_{\beta}(J_2, p)\}_{\beta \in B}$ be sets of  simplicial chains which represent bases of the free abelian groups $\widetilde H_p\big((\sk_1)_{J_1}\big)$, $\widetilde H_p\big((\sk_2)_{J_2}\big)$ respectively. Let $\{\chi_{\alpha}(J_1, p)\}_{\alpha \in A}$, $\{\chi_{\beta}(J_2, p)\}_{\beta \in B}$ be the images of these bases under the map $C_{p-1}((\sk_l)_{J_l}) \to \mathcal C_{p + |J_l|}(\mathcal Z_{\sk_l}), l = 1, 2$, of \hyperref[homology]{Theorem~\ref{homology}}. Considering the same bases as elements of $\widetilde H_*(\sk_J)$, we obtain that they map under the homomorphism $C_{p-1} \to \mathcal C_{p+|J|}$ to the cellular chains
\begin{align}\label{chain1}
    \prod\limits_{j\in J\setminus J_2} S_{j}\cdot\{\chi_{\alpha}(J_1, p)\}_{\alpha \in A}\qquad \text{and}\qquad \prod\limits_{j\in J\setminus J_1} S_{j}\{\chi_{\beta}(J_2, p)\}_{\beta \in B}
\end{align}
respectively. When $J_1 \neq\varnothing, J_2\neq\varnothing$, but $J\cap I = \varnothing$, we get a new generator of $\widetilde H_0(\sk_J)$ which is represented by a simplicial chain $j_1 + j_2$ with $j_1 \in J_1$ and $j_2 \in J_2$; this one is different from generators of homology of $\sk_{J_1}$ and $\sk_{J_2}$. The corresponding cellular chain in $\mathcal C_*(\zk)$ is
\begin{align}\label{chain2}
\prod\limits_{j \neq j_1, j_2} S_j\cdot (D_{j_1}S_{j_2} + S_{j_1}D_{j_2}).
\end{align}

Let $\{w_{\alpha}(J_1, p)\}_{\alpha \in A}$, $\{w_{\beta}(J_2, p)\}_{\beta \in B}$ be the Whitehead products corresponding to the bases $\{\sigma_{\alpha}(J_1, p)\}_{\alpha \in A}$, $\{\sigma_{\beta}(J_2, p)\}_{\beta \in B}$. The Hurewicz images of the products
\begin{align*}
    &\Big[\mu_{k_1}, \big[\mu_{k_2},\big[\ldots\big[\mu_{k_{r-1}},[\mu_{k_r}, w_{\alpha}(J_1, p)]\big]\ldots\Big] \qquad \text{and} &&\Big[\mu_{l_1}, \big[\mu_{l_2},\big[\ldots\big[\mu_{l_{s-1}},[\mu_{l_s}, w_{\beta}(J_2, p)]\big]\ldots\Big] \\
    &\text{for } J\setminus J_2 = \{k_1,\ldots, k_r\} &&\text{for } J\setminus J_1 = \{l_1,\ldots, l_s\}
\end{align*}
are represented by chains~\eqref{chain1}. Chain~\eqref{chain2} represents the Hurewicz image of the  product 
\[
\Big[\mu_{j_3}, \big[\mu_{j_4},\big[\ldots\big[\mu_{j_{|J|}},[\mu_{j_1}, \mu_{j_2}]\big]\ldots\Big], J = \{j_1,\ldots, j_{|J|}\}.
\]

The wedge sum of the Whitehead products above
\begin{align*}
    \Big(\bigvee_{\substack{J_1, J_2 \\ J_1\cap I = J_2\cap I}}\Big(\bigvee_{p \geqslant 0}\big(\bigvee_{\alpha\in A}S^{p + |J|+1}\vee\bigvee_{\beta\in B}S^{p + |J| + 1}\big) \Big)\Big) \vee
    \bigvee_{\substack{J_1, J_2 \neq\varnothing \\ J_1\cap I = J_2\cap I = \varnothing}} S^{|J| + 1} \longrightarrow \mathcal Z_{\sk}
\end{align*}
induces an isomorphism in the homology. As all the spaces involved are simply connected, it is a homotopy equivalence.
\end{proof}

\begin{figure}[!h]
	\begin{minipage}[h]{0.49\linewidth}	
		\begin{center}
			\begin{tikzpicture}
			\coordinate (A1) at ( 0cm, 4cm);
			\coordinate (A2) at ( 0cm,-4cm);
			\coordinate (A3) at (-3cm, 0cm);
			\coordinate (A4) at ( 1cm,-1cm);
			\coordinate (A5) at ( 3cm,  0cm);
			
			\fill[gray!20, opacity = 0.3] (A1) -- (A4) -- (A3) -- cycle;
			\fill[gray!20, opacity = 0.3] (A1) -- (A4) -- (A5) -- cycle;
			\fill[gray!20, opacity = 0.7]  (A2) -- (A4) -- (A5) -- cycle;
			\fill[gray!20, opacity = 0.5]  (A2) -- (A4) -- (A3) -- cycle;
			
			\foreach \j in {3, 4, 5}	
			{
				\draw[thick] (A1) -- (A\j) ;
				\draw[thick] (A2) -- (A\j) ;
			}
			\draw[thick] (A3) -- (A4) -- (A5);
			\draw[thick, dashed] (A1) -- (A2); 
			\draw[thick, dashed] (A3) -- (A5); 
			
			\foreach \i in {1, 2, ..., 5}
			{
				\draw[fill=black] (A\i) circle (0.2em);
			}
			
			\draw[fill = black] (0cm, -0.3cm) circle (0.1em);
			
			\draw (A1) node[above right] {\textbf 1};
			\draw (A2) node[right] 	      {\textbf 2};
			\draw (A3) node[above left] {\textbf 3};
			\draw (A4) node[below right] {\textbf 4};
			\draw (A5) node[above right] {\textbf 5};
			
			\end{tikzpicture}
			\caption{$\sk = \sj_1(\partial\Delta^2)$}\label{pic1}
		\end{center}
	\end{minipage}
	\hfill
	\begin{minipage}[h]{0.49\linewidth}	
		\begin{center}
			\begin{tikzpicture}
			\coordinate (A1) at ( 0cm, 4cm);
			\coordinate (A2) at ( 0cm,-4cm);
			\coordinate (A3) at (-3cm, 0cm);
			\coordinate (A4) at ( 1cm,-1cm);
			\coordinate (A5) at ( 3cm,  0cm);
			
			\fill[gray!20, opacity = 0.3] (A1) -- (A4) -- (A3) -- cycle;
			\fill[gray!20, opacity = 0.3] (A1) -- (A4) -- (A5) -- cycle;
			\fill[gray!20, opacity = 0.7]  (A2) -- (A4) -- (A5) -- cycle;
			\fill[gray!20, opacity = 0.5]  (A2) -- (A4) -- (A3) -- cycle;
			
			\foreach \j in {3, 4, 5}	
			{
				\draw[thick] (A1) -- (A\j) ;
				\draw[thick] (A2) -- (A\j) ;
			}
			\draw[thick] (A3) -- (A4) -- (A5);
			\draw[thick, dashed] (A1) -- (A2); 
			\draw[thick, dashed] (A3) -- (A5); 
			\draw[pattern=dots] (A1) -- (A2) -- (A3) -- cycle;
			
			\foreach \i in {1, 2, ..., 5}
			{
				\draw[fill=black] (A\i) circle (0.2em);
			}
			
			\draw[fill = black] (0cm, -0.3cm) circle (0.1em);
			
			\draw (A1) node[above right] {\textbf 1};
			\draw (A2) node[right] 	      {\textbf 2};
			\draw (A3) node[above left] {\textbf 3};
			\draw (A4) node[below right] {\textbf 4};
			\draw (A5) node[above right] {\textbf 5};
			
			\end{tikzpicture}
			\caption{$\mathcal L = \sj_1(\partial\Delta(3, 4, 5))\cup\{1, 2, 3\}$}\label{pic2}
		\end{center}
	\end{minipage}
\end{figure}	

\begin{theorem}
\label{soper}
	Let $\sk \in W_{\Delta}$ be a simplicial complex. Then the simplicial complex
	\begin{align*}
	\mathcal {J}_n(\sk) = (\partial\Delta^n*\sk)\cup \Delta^n
	\end{align*}
	belongs to $W_{\Delta}$.
\end{theorem}

Note that $\sj_n(\sk) \simeq \Sigma^n(\sk)\vee S^n$.

The case $\sk = \partial\Delta^2$ and $n = 1$ is shown in \hyperref[pic1]{Figure~
\ref{pic1}}.

\begin{proof}
	Denote $\mathcal{L} = \mathcal {J}_n(\sk)$ and let $V(\sk) = I$, $V(\Delta^n) = I_1$ be the sets of vertices. By \hyperref[homology]{Theorem~\ref{homology}}, homology of $\mathcal{Z_L}$ comes from homology of full subcomplexes $\mathcal L_{J_1, J} = ((\partial\Delta^n)_{J_1}*\sk_J)\cup \Delta_{J_1}^n$ where $J_1 \subset I_1, J \subset I$.

	If $J_1\subsetneq I_1$ is a nonempty proper subset, then the complex $\mathcal L_{J_1, J}$ is topologically contractible. So, in this case $\widetilde H_*(\mathcal L_{J_1, J}) = 0$. If $J_1 = \varnothing$, then the corresponding full subcomplex is $\mathcal L_{\varnothing, J} = \sk_J$. Hence, $\widetilde H_*(\mathcal L_{\varnothing, J}) \cong \widetilde H_*(\sk_J)$. Finally, when $J_1 = I_1$, we have
\begin{align}\label{formula}
    \mathcal L_{I_1, J} = (\partial\Delta^n*\sk_J)\cup \Delta^n \simeq \Sigma^n(\sk_J)\vee S^n.
\end{align}
Hence, we get $\widetilde H_*(\mathcal L_{I_1, J}) \cong \widetilde H_{*-n} (\sk_J)\oplus \widetilde H_*(S^n)$, where the generator of the second summand is represented by the boundary of the $(n+1)$-simplex $\Delta(I_1, j)$.

We shall show that all generators of $H_*(\mathcal{Z_L})$ are represented by the cellular chains~\eqref{f_hur_im_2}, and therefore are the Hurewicz images of iterated higher Whitehead products by \hyperref[hur_im_2]{Lemma~\ref{hur_im_2}}. The generator of $H_*(\mathcal{Z_L})$ corresponding to the wedge summand $S^n$ in~\eqref{formula} is represented by the cellular chain
\begin{align*}
    \prod\limits_{k = 1}^{m-1}S_{j_k}\cdot\Big(D_{i_1}\cdots D_{i_{n+1}}S_j + (\sum\limits_{k = 1}^{n+1}D_{i_1}\cdots S_{i_k}\cdots D_{i_{n+1}})D_j\Big),
\end{align*}
where $I_1 = \{i_1,\ldots, i_{n+1}\}$ and $J = \{j_1,\ldots, j_{m-1}, j\}$. It is the Hurewicz image of the Whitehead product
\begin{align*}
	\Big[\mu_{j_1}, \big[\mu_{j_2},\big[\ldots\big[\mu_{j_{m-1}},[\mu_{i_1},\ldots, \mu_{i_{n+1}}, \mu_j]\big]\ldots\Big].
\end{align*}
Finally, every generator of $H_*(\mathcal{Z_L})$ corresponding to the wedge summand $\Sigma^n(\sk_J)$ in~\eqref{formula} is represented by the cellular chain
\begin{align}\label{chain}
	\Big(\sum\limits_{k=1}^{n+1}D_{i_1}\cdots S_{i_k}\cdots D_{i_{n+1}}\Big)h(w),
\end{align}
where $w$ is the Whitehead product which maps to corresponding generator of $\widetilde H_{*-n}(\sk_J)\subset H_*(\mathcal{Z_K})$. The chain~\eqref{chain} represents the class $h([\mu_{i_1},\ldots,\mu_{i_{n+1}}, w])$.

The wedge sum of the Whitehead products described above is a map from a wedge of spheres to $\mathcal{Z_L}$ that induces an isomorphism of homology groups. Hence, it is a homotopy equivalence.
\end{proof}

\begin{theorem}
\label{exist}
	For any iterated higher Whitehead product
	\begin{align}\label{product}
		w = \Big[\mu_{i_{01}}, \ldots, \mu_{i_{0p_0}}, [ \ldots [\mu_{i_{n 1}}, \ldots, \mu_{i_{n p_n}}] \ldots\Big],
	\end{align}
	there exists a simplicial complex $\sk \in W_{\Delta}$ that realizes $w$.
\end{theorem}
\begin{proof}
Consider the simplicial complex $\mathcal L = \sj_{p_0-1}\circ\sj_{p_1-1}\circ\cdots\circ\sj_{p_{n-1}-1}(\partial\Delta^{p_n - 1})$. By \hyperref[soper]{Theorem~\ref{soper}}, the complex $\mathcal L$ belongs to $W_{\Delta}$. The maximal dimensional sphere in the wedge $\mathcal{Z_L}$ has dimension $2(p_0 + \cdots +p_n) - (n+1)$ and it is realized by $w$, as shown by considering the Hurewicz homomorphism.
\end{proof}

\begin{table}[!h]
	\renewcommand{\arraystretch}{1.7}
	\begin{center}
		\begin{tabular}{||c||c|c|c||}
			\hline
			\multirow{4}{*}{$H_5(\mathcal{Z_K})$}
			& $\Z$ & $D_3D_4S_5 + D_3S_4D_5 + S_3D_4D_5$ & $[\mu_3, \mu_4, \mu_5]$ \\
			\cline{2-4}
			& $\Z$ & $D_1D_2S_3 + D_1S_2D_3 + S_1D_2D_3$ & $[\mu_1, \mu_2, \mu_3]$ \\
			\cline{2-4}
			& $\Z$ & $D_1D_2S_4 + D_1S_2D_4 + S_1D_2D_4$ & $[\mu_1, \mu_2, \mu_4]$ \\
			\cline{2-4}
			& $\Z$ & $D_1D_2S_5 + D_1S_2D_5 + S_1D_2D_5$ & $[\mu_1, \mu_2, \mu_5]$\\
			\hline
			\multirow{3}{*}{$H_6(\mathcal{Z_K})$}
			& $\Z$ & $S_4(D_1D_2S_3 + D_1S_2D_3 + S_1D_2D_3)$ & $\big[\mu_4,[\mu_1, \mu_2, \mu_3]\big]$ \\
			\cline{2-4}
			& $\Z$ & $S_5(D_1D_2S_3 + D_1S_2D_3 + S_1D_2D_3)$ & $\big[\mu_5,[\mu_1, \mu_2, \mu_3]\big]$ \\
			\cline{2-4}
			& $\Z$ & $S_5(D_1D_2S_4 + D_1S_2D_4 + S_1D_2D_4)$ & $\big[\mu_5,[\mu_1, \mu_2, \mu_4]\big]$\\
			\hline
			$H_7(\mathcal{Z_K})$ &	$\Z$ & $S_5S_4(D_1D_2S_3 + D_1S_2D_3 + S_1D_2D_3)$ & $\Big[\mu_4,\big[\mu_5,[\mu_1,\mu_2,\mu_3]\big]\Big]$
			\\
			\hline
			$H_8(\mathcal{Z_K})$ &	$\Z$ &	$(D_1S_2+S_1D_2)(D_3D_4S_5 + D_3S_4D_5 + S_3D_4D_5)$ & $\big[\mu_1,\mu_2,[\mu_3,\mu_4,\mu_5]\big]$
			\\
			\hline
		\end{tabular}
		\caption{Homology of $\mathcal{Z_K}$ for $\sk = \sj_1(\partial\Delta^2)$ (see~\hyperref[pic1]{Figure~\ref{pic1}})}\label{t1}
	\end{center}
\end{table}
\begin{exm}\label{j_0}
	The Whitehead product $\big[\mu_1,\mu_2,[\mu_3,\mu_4,\mu_5]\big]$ is realised by the complex $\sk$ which is the minimal triangulation of a $2$-sphere with diameter (see~\hyperref[pic1]{Figure~\ref{pic1}}). Note that $\sk = \sj_1\big(\partial\Delta(3, 4, 5)\big)$, so $\sk\in W_{\Delta}$.

	Using \hyperref[homology]{Theorem~\ref{homology}} and \hyperref[hur_im_2]{Lemma~\ref{hur_im_2}} we identify homology of $\zk$, cellular chains representing homology generators and Whitehead products which map to the generators under the Hurewicz homomorphism (see~\hyperref[t1]{Table~\ref{t1}}).

	The wedge sum of the Whitehead products from the right column of \hyperref[t1]{Table~\ref{t1}} gives a map $(S^5)^{\vee 4}\vee(S^6)^{\vee 3}\vee S^7 \vee S^8 \to \mathcal{Z_K}$ that induces an isomorphism of homology groups. Thus, we have $\mathcal{Z_K} \simeq (S^5)^{\vee 4}\vee(S^6)^{\vee 3}\vee S^7 \vee S^8$.
	
	Note that this is the first example of $\sk$ with a nontrivial iterated Whitehead product in which one of the arguments of a higher product is again a higher product.
\end{exm}

\section{The smallest complex realizing a given Whitehead product}\label{min}
\begin{theorem}
	\label{least}
	The simplicial complex $\sj_{p-1}(\partial\Delta^{q-1}) \in W_{\Delta}$ \textup{(see~\hyperref[soper]{Theorem~\ref{soper}})} is the smallest complex realizing the product
	\begin{align}
	\label{product2}
		\big[\mu_{i_1}, \ldots, \mu_{i_p}, [\mu_{j_1}, \ldots, \mu_{j_q}]\big].
	\end{align}
\end{theorem}

\begin{proof} Assume that $\sk$ realizes the product~\eqref{product2}.

For the product~\eqref{product2} to be defined it is necessary that the product $[\mu_{j_1}, \ldots, \mu_{j_q}]$ is defined and the products $[\mu_{i_1}, \ldots, \mu_{i_p}]$ and
\begin{align}
\label{product3}
	\big[\mu_{i_1}, \ldots \widehat\mu_{i_k},\ldots,\mu_{i_p},[\mu_{j_1}, \ldots, \mu_{j_q}]\big], \quad \text{for $k = 1,\ldots, p$}
\end{align}
are trivial.
	
For the existence of $[\mu_{j_1}, \ldots, \mu_{j_q}]$ it is necessary that every $[\mu_{j_1},\ldots\widehat \mu_{j_k},\ldots, \mu_{j_q}]$ with $k = 1,\ldots, q$ is trivial, that is $\{j_1,\ldots \widehat{j_k} ,\ldots, j_q\}\in\sk$ is a simplex. Thus, the existence of the $[\mu_{j_1}, \ldots, \mu_{j_q}]$ gives an inclusion $\partial \Delta^{q - 1} \hookrightarrow \sk$ and the triviality of  $[\mu_{i_1}, \ldots, \mu_{i_p}]$ gives an inclusion $\Delta^{p-1} = \{i_1, \dots, i_p\} \hookrightarrow \sk$.
	
We shall show that the triviality of products~\eqref{product3} is equivalent to the existence of inclusions
\begin{align}\label{incl}
    \{i_1,\ldots\widehat i_k,\ldots, i_p\}*\partial\Delta^{q-1}\hookrightarrow \sk,\quad \text{for $k = 1,\ldots, p$.}
    \end{align}
	We can assume $k = p$ without loss of generality. Since $\big[\mu_{i_1}, \ldots,\mu_{i_{p-1}},[\mu_{j_1}, \ldots, \mu_{j_q}]\big]$ is trivial we have an inclusion
	\begin{align*}
    D_{i_1}^2\times\cdots\times D_{i_{p-1}}^2\times\left(\bigcup\limits_{k = 1}^q D_{j_1}^2\times\cdots\times S_{j_k}^1\times\cdots\times D_{j_q}^2\right) \hookrightarrow \mathcal{Z_K},
    \end{align*}
    see the proof of the case $n = 1$ in \hyperref[hur_im_2]{Lemma~\ref{hur_im_2}}. Therefore, we have an inclusion~\eqref{incl} for $k = p$.

 	It follows that $\sj_{p-1}(\partial\Delta^{q-1})$  embeds into any simplicial complex $\sk$ that realizes the product~\eqref{product2}. By \hyperref[exist]{Theorem~\ref{exist}}, the complex $\sk = \sj_{p-1}(\partial\Delta^{q-1})$ realizes the product~\eqref{product2}, so it is the smallest complex with this property.
\end{proof}

The next theorem is proved similarly.

\begin{theorem}
The simplicial complex $\sj_{p_0-1}\circ\sj_{p_1-1}\circ\dots\circ\sj_{p_{n-1}-1}(\partial\Delta^{q-1}) \in W_{\Delta}$ \textup{(see~\hyperref[soper]{Theorem~\ref{soper}})} is the smallest complex realizing the product
\[
	\big[\mu_{i_{01}}, \ldots, \mu_{i_{0p_0}},[\dots [\mu_{i_{n1}}, \ldots, \mu_{i_{np_n}}]\dots]\big].
\]
\end{theorem}

In \hyperref[least]{Theorem~\ref{least}} we have shown that if the simplicial complex $\mathcal L$ realizes the product~\eqref{product2} than there exists an embeding $\mathcal J_{p-1}(\partial\Delta^{q-1}) \hookrightarrow \mathcal L$. However, the next example shows that $\mathcal J_{p-1}(\Delta^{q-1})$ is not necessarily a full subcomplex of $\mathcal L$.

\begin{table}[h]
\renewcommand{\arraystretch}{1.3}
	\begin{center}
		\begin{tabular}{||c||c|c|c||}
			\cline{1-4}
			\multirow{3}{*}{$H_5(\mathcal{Z_K})$}
			& $\Z$ & $D_3D_4S_5 + D_3S_4D_5 + S_3D_4D_5$ & $[\mu_3, \mu_4, \mu_5]$ \\
			\cline{2-4}
			& $\Z$ & $D_1D_2S_4 + D_1S_2D_4 + S_1D_2D_4$ & $[\mu_1, \mu_2, \mu_4]$ \\
			\cline{2-4}
			& $\Z$ & $D_1D_2S_5 + D_1S_2D_5 + S_1D_2D_5$ & $[\mu_1, \mu_2, \mu_5]$\\
			\cline{1-4}
			$H_6(\mathcal{Z_K})$ & $\Z$ & $S_5(D_1D_2S_4 + D_1S_2D_4 + S_1D_2D_4)$ & $\big[\mu_5,[\mu_1, \mu_2, \mu_4]\big]$\\
			\hline
			$H_8(\mathcal{Z_K})$ &	$\Z$ & $(D_1S_2+S_1D_2)(D_3D_4S_5 + D_3S_4D_5 + S_3D_4D_5)$ & $\big[\mu_1,\mu_2,[\mu_3,\mu_4,\mu_5]\big]$
			\\
			\hline
		\end{tabular}
		\caption{Homology of $\mathcal{Z_L}$ for $\mathcal L = \sj_1(\partial\Delta_{3, 4, 5}^2)\cup\{1, 2, 3\}$ (see~\hyperref[pic2]{Figure~\ref{pic2}})}\label{t2}
	\end{center}
\end{table}
	
\begin{exm}
	Consider the simplicial complex $\mathcal L = \sj_1\big(\partial\Delta(3, 4, 5)\big)\cup\Delta(1, 2, 3)$, see~\hyperref[pic2]{Figure~\ref{pic2}}.	An argument similar to \hyperref[j_0]{Example~\ref{j_0}} shows that $\mathcal{Z_K} \simeq (S^5)^{\vee 3}\vee S^6 \vee S^8$, see~\hyperref[t2]{Table~\ref{t2}}. Here $S^8$ is a lift of the product $\big[\mu_1,\mu_2,[\mu_3,\mu_4,\mu_5]\big]$. One can see that $\sj_1\big(\partial\Delta(3, 4, 5)\big)$ is not a full subcomplex of $\mathcal L$.
\end{exm}

\section{An example of unrealizability} \label{non}

In this section we give an example of a simplicial complex $\sk$ such that the corresponding moment-angle complex $\zk$ is homotopy equivalent to a wedge of spheres, but $\sk \notin W_{\Delta}$. In other words, there is a sphere in the wedge which is not realizable by any linear combination of iterated higher Whitehead products (in the sense of \hyperref[wdelta]{Definitioin~\ref{wdelta}}). This implies that the answer to the problem \cite[Problem 8.4.5]{TT} is negative.

\begin{prop}
	Let $\sk$ be the simplicial complex $(\partial\Delta^2*\partial\Delta^2)\cup\Delta^2\cup\Delta^2$. The moment-angle complex $\zk$ is homotopy equivalent to a wedge of spheres.
\end{prop}
\begin{proof}
	Note that the $2$-skeleton of $\sk$ coincides with the $2$-skeleton of a $5$-dimensional simplex. Therefore, $\zk$ is $6$-connected.
	Note that $|\sk| \simeq S^3\vee S^2\vee S^2$.
	Using \hyperref[homology]{Theorem~\ref{homology}} one can describe homology of $\mathcal{Z_K}$  (see~\hyperref[t2]{Table~\ref{t3}}).
	\begin{table}[h]
    	\renewcommand{\arraystretch}{1.3}
    	\begin{center}
    		\begin{tabular}{||c|c||}
                \hline
                $H_7(\mathcal{Z_K})$ & $\Z^6$ \\
                \hline
                $H_8(\mathcal{Z_K})$ & $\Z^6$ \\
                \hline
                $H_9(\mathcal{Z_K})$ & $\Z^2$ \\
                \hline
                $H_{10}(\mathcal{Z_K})$ & $\Z$ \\
                \hline
			\end{tabular}
    	\caption{Homology of $\mathcal{Z_K}$ for $\sk = (\partial\Delta^2*\partial\Delta^2)\cup\Delta^2\cup\Delta^2$}\label{t3}
    	\end{center}
    \end{table}

    By \cite{BBCG} (see also~\cite[Corollary 8.3.6]{TT}), there is a homotopy equivalence
    \begin{align*}
    	f\colon (S^8)^{\vee 6}\vee (S^9)^{\vee 6}\vee (S^{10})^{\vee 2}\vee S^{11} \xrightarrow{\simeq} \Sigma\mathcal{Z_K}.
    \end{align*}
    Denote $X = (S^7)^{\vee 6}\vee (S^8)^{\vee 6}\vee (S^9)^{\vee 2}\vee S^{10}$. As both spaces $X$ and $\zk$ are $6$-connected, by the Freudental Theorem, the suspension homomorphism $\Sigma\colon \pi_n \to \pi_{n+1}$ for $X$ and $\mathcal{Z_K}$ is an isomorphism for $n < 13$. Consider the following commutative diagram for $n < 13$:
    \begin{center}
		\begin{tikzpicture}[every node/.style={midway}]
		\matrix[column sep={5cm,between origins},row sep={2cm}] at (0,0) {
			\node(A)   {$\pi_{n+1}(\Sigma X)$}; & \node(B) {$\pi_{n+1}(\Sigma \zk)$}; \\
			\node(C)   {$\pi_{n  }(X)$}; & \node(D) {$\pi_{n  }(\zk)$.}; \\};
		\draw[->] (C) -- (A) node[anchor=east] {$\cong$} node[anchor=west] {$\Sigma_X$};
		\draw[->] (A) -- (B) node[anchor=south]  {$f_*$};
		\draw[->] (D) -- (B)  node[anchor=east] {$\cong$} node[anchor=west] {$\Sigma_{\zk}$};
		\draw[dashed, ->](C) -- (D) node[anchor=south] {$\Sigma_{\zk}^{-1}\circ f_* \circ \Sigma_X$};
		\end{tikzpicture}
	\end{center}
	The class $[i^n_j]$ of the inclusion of the $j$-th $n$-sphere $i_j^n\colon S^n \hookrightarrow X$ maps to the class of a map $S^n\to \zk$ under the composite $\Sigma_{\zk}^{-1}\circ f_* \circ \Sigma_X$. The wedge sum of these maps gives a map $g\colon X\to\zk$. By construction, $\Sigma g\colon \Sigma X \to \Sigma\zk$ induces an isomorphism in homology. Thus, $g$ also induces an isomorphism in homology, so it is a homotopy equivalence.
\end{proof}

\begin{prop}
	The sphere $S^{10}\subset \zk$ cannot be realized by any linear combination of general iterated higher Whitehead products.
\end{prop}

\begin{proof}
By dimensional reasons, if there is a general iterated higher Whitehead realizing the sphere $S^{10}$, then it contains exactly two nested brackets.
The internal bracket may have size 2, 3, 4 or 5.
Since the 2-skeleton of $\sk$ coincides with the 2-skeleton of $\Delta^5$, all Whitehead 2- and 3-products are trivial. We are left with the following two iterated products with the internal bracket of size 5 and~4:
$$
[\mu_{i_1}, [\mu_{i_2}, \mu_{i_3}, \mu_{i_4}, \mu_{i_5}, \mu_{i_6}]], \quad
[\mu_{i_1}, \mu_{i_2}, [\mu_{i_3}, \mu_{i_4}, \mu_{i_5}, \mu_{i_6}]]. 
$$
The first product is not defined because $\sk$ does not contain $\partial\Delta(i_2, i_3, i_4, i_5, i_6)$. For the second product, without loss of generality we can consider two cases $\{i_1, i_2\} = \{1, 2\}$ and $\{i_1, i_2\} = \{1, 4\}$. In both cases the product is not defined because $\sk$ does not contain $\Delta(2,4,5,6)$.
\end{proof}


\begin{thebibliography}{GPTW}
	\bibitem[BBCG]{BBCG}
	A.~Bahri, M.~Bendersky, F.R.~Cohen, S.~Gitler.
	\textit{The polyhedral product functor: a method of decomposition for moment-angle complexes, arrangements and related spaces.} Adv. Math. 225 (2010), no. 3, 1634--1668.
	
	\bibitem[BP]{TT}
	V.~Buchstaber, T.~Panov.
	\textit{Toric Topology.} Math. Surv. and Monogr.,~204, Amer. Math. Soc., Providence, RI, 2015.
	
	\bibitem[GT]{gt}
	J. Grbi\'c, S. Theriault.
	\textit{Homotopy theory in toric topology.} Russian Math. Surveys, 71:2 (2016), 185--251.
	
	
	\bibitem[GPTW]{gptw}
	J.~Grbi\'c,	T.~Panov, S.~Theriault, J.~Wu.
	\textit{Homotopy types of moment-angle complexes for flag complexes.} Trans. Amer. Math. Soc.~\textbf{368} (2016), no.~9, 6663--6682.
	
	\bibitem[Ha]{ha}
	K.A.~Hardie.
	\textit{Higher Whitehead products.} Quart. J. Math. Oxford Ser. (2) 12 1961 241--249.
	
	\bibitem[IK]{ki}
	K.~Iriye, D.~Kishimoto.
	\textit{Polyhedral products for shifted complexes and higher Whitehead products}. arXiv:1505.04892.
	
	\bibitem[PR]{pr}
	T.~Panov, N.~Ray.
	\textit{Categorical aspects of toric topology.} Toric topology, 293--322, Contemp. Math., 460, Amer. Math. Soc., Providence, RI, 2008.	
	
	
	\bibitem[Ve]{veryovkin16}
	Ya.~Veryovkin.
	\textit{Pontryagin algebras of some moment-angle-complexes}. Dal'nevost. Mat. Zh. (2016), no.~1, 9--23 (in Russian); arXiv:1512.00283.
\end{thebibliography}
\end{document}